\documentclass[a4paper,12pt]{article}
\usepackage{latexsym}
\usepackage{amsfonts}
\usepackage{amssymb}
\usepackage{amsmath}
\usepackage{graphicx}
\usepackage{makeidx}
\usepackage{textcomp}
\usepackage{multirow}
\usepackage{amsthm}

\newtheorem{theorem}{Theorem}[section]
\newtheorem{lemma}[theorem]{Lemma}
\newtheorem{proposition}[theorem]{Proposition}

\newtheorem{conjecture}{Conjecture}
\theoremstyle{definition} \newtheorem{definition}[theorem]{Definition}
\theoremstyle{definition} 
\theoremstyle{definition} \newtheorem{remark}{Remark}

\title{A Theorem of Siemons and Wagner}
\author{Paul Bradley\\
\small Jacobs CH2M\\[-0.8ex]
\small Birmingham, U.K.\\
\small\tt paulmbradley82@gmail.com\\
}

\begin{document}
\maketitle
\vskip 0.4 true cm

\begin{center}{\textbf{Acknowledgements}}
\end{center}

\vskip 0.4 true cm

\begin{abstract}
In \cite{SandW} Siemons and Wagner described a relationship between the lengths of $G$-orbits on subsets of a $G$-set $\Omega$. They highlighted the situation where $\Delta \subset \Omega$ with $|\Delta|=k,$ and $|\Delta^G|>|\Sigma^G|$ for all $(k+1)$-subsets, $\Sigma$ of $\Omega$ where $\Delta\subset\Sigma$. They went on to classify all primitive groups with this property for $k=2$. Here we address some questions about primitive permutation groups satisfying this property when $k=3$ and list all $3$-homogeneous groups satisfying this condition. 

\end{abstract}

\section{Introduction}

Let $G$ be a permutation group acting on finite set $\Omega$ then $G$ has an induced action on $\Omega_k$ the set of subsets of $\Omega$ of size $k$. The number of orbits has been been the subject of numerous papers, not least of these is the paper of Livingstone and Wagner~\cite{Livingstone and Wagner}. However, similar results for $G$-orbit lengths are much less common. The main contributions coming from Siemons and Wagner \cite{SandW}, \cite{SW2} and Mnukhin \cite{Munchkin}. We note that we are interested in the finite case and acknowledge much work has been done on the infinite case.

In \cite{SandW} Siemons and Wagner describe a relationship between the lengths of $G$-orbits on subsets of a $G$-set $\Omega$.
They proved the following;
\begin{theorem}[Siemons and Wagner]\label{SW1}
Let $G$ be a transitive permutation group on a finite set $\Omega$ and let $\Delta$ be a subset of $\Omega$ of cardinality $k$ such that $|\Delta^G|>|\Sigma^G|$ for every subset $\Sigma$ containing $\Delta$ of cardinality $k+1$. Then
$$k+1 \ge |\Delta^{G_\Sigma}|>|\Sigma^{G_\Delta}|\ge 1.$$

Furthermore, if $k\ge 2$ then either
\begin{enumerate}
\item every $2$-element subset of $\Omega$ will appear in some $G$ image of $\Delta$ or
\item $G$ is imprimitive with blocks of imprimitivity $\mathfrak{B}_1,...,\mathfrak{B}_r$ ($1<|\mathfrak{B}_i |<|\Omega |$) each intersecting $\Delta$ in at most $1$ point such that every $2$-element subset of the form $\{\alpha_i, \alpha_j\}$ with $\alpha_i \in \mathfrak{B}_i \ne \mathfrak{B}_j \ni \alpha_j$ is contained in some $G$ image of $\Delta$.
\end{enumerate}

\end{theorem}

They also proved that if $G$ is primitive and $k=2$ then $G\cong L_2(5)$ with $G$ acting on $6$ points. The aim of this paper is to answer some questions regarding primitive groups with the Siemons Wagner property and more specifically the case $k=3$. 

The Siemons Wagner property appears to be rare amongst primitive groups, indeed we can identify all cases where $\Omega$ has cardinality $n\leq 24$.

\begin{proposition}\label{PPPSW}
Let $G$ be a primitive permutation group acting on a set $\Omega$ of degree $n<25$. If there exists a $k$-subset $\Delta\subset\Omega$ such that $|\Delta^G|>|\Sigma^G|$ for any $k+1$-subset $\Sigma\subset\Omega$ where $\Delta \subset \Sigma$, then $G$ is in Table~\ref{PPSW}.
\begin{table}[htbp]
  \centering
  \begin{tabular}{|c|c|c|}\hline
    Group		&		Degree		&		k		\\ \hline
$	L_2(5)	$	&	$	6	$	&	$	2	$	\\
$	L_2(7)	$	&	$	8	$	&	$	3	$	\\
$	PGL(2,7)	$	&	$	8	$	&	$	3	$	\\
$	L_2(9)	$	&	$	10	$	&	$	4	$	\\
$	Sym(6)	$	&	$	10	$	&	$	4	$	\\
$	L_2(11)	$	&	$	12	$	&	$	5	$	\\
$	PGL(2,11)	$	&	$	12	$	&	$	5	$	\\
$	L_2(13)	$	&	$	14	$	&	$	6	$	\\
$	Alt(7)	$	&	$	15	$	&	$	6	$	\\
$	ASL(2,4)	$	&	$	16	$	&	$	6	$	\\
$	Alt(7)\ltimes (\mathbb{Z}_2)^4	$	&	$	16	$	&	$	7	$	\\
$	L_2(16)	$	&	$	17	$	&	$	5	$	\\
$	L_3(4)	$	&	$	21	$	&	$	6	$	\\
$	M(22)	$	&	$	22	$	&	$	10	$	\\
$	M(23)	$	&	$	23	$	&	$	10	$	\\
$	M(24)	$	&	$	24	$	&	$	11	$	\\ \hline
 \end{tabular}
\caption{Primitive permutation groups satisfying Siemons Wagner property}
\label{PPSW}
\end{table}
\end{proposition}

\begin{proof}
We can compile Table~\ref{PPSW} by direct calculation with the database of primitive groups in \textsc{Magma}. The code used to test for the Siemons Wagner property is given at the end of this paper.
\end{proof}
\begin{remark}
We note that the database of primitive groups used in the \textsc{Magma} calculations was produced by Roney-Dougal in \cite{Colva} and Sims~\cite{SimsPrim}.
\end{remark}

\begin{definition}
Let $G$ be a group acting on a set $\Omega$ with cardinality $n$. Denote by $\Omega_k$ the set of $k$-subsets of $\Omega$. We note that $|\Omega_k|=\binom{n}{k}$ and that $G$ has an induced action on $\Omega_k$. If $G$ is transitive on $\Omega_k$ we call $G$ $k$-homogeneous. 

If $G$ is a permutation group on a finite set $\Omega$ and $\Delta$ is a subset of $\Omega$ of cardinality $3$ such that $|\Delta^G|>|\Sigma^G|$ for every subset $\Sigma$ containing $\Delta$ of cardinality $4$. Then we say that $G$ satisfies condition $\star$.
\end{definition}

\section{General Results}

Here we present a few preliminary results looking at the case when $k=3$ and take some steps towards classifying primitive groups with this property for $k=3$.

In general we will make use of the following simple but effective lemma.

\begin{lemma}\label{ud lemma}
Let $G$ be a permutation group acting on $n$ points, let $\Sigma^G$ be an $G$-orbit of a $(k+1)$-subset, and $\Delta^G$ and $G$-orbit of a $k$-subset $\Delta$. If $\Sigma \supset \Delta$ then letting $d = |\{\alpha \in \Sigma \ | \ \Sigma\setminus\{\alpha\} \in \Delta^G\}|$ and $u=|\{\beta \in \Omega\ | \ \Delta \cup \{\beta\} \in \Sigma^G\}|$ then
$$d|\Sigma^G|=u|\Delta^G|.$$
\end{lemma}
\begin{proof}
We form a graph with vertex set the elements of $\Sigma^G$ and $\Delta^G$, then we draw edge $(s,t)$ if and only if $t \subset s$. Then the number of edges is equal to $d|\Sigma^G|$ and $u|\Delta^G|$.
\end{proof}

We now are able to prove the following initial results regarding primitive permutation groups satisfying $\star$ without any further restrictions.

\begin{proposition}\label{maxsize}
Let $G$ be a primitive permutation group acting on a set $\Omega$ of cardinality $n\ge 8$.
Let $\Delta$ be a $3$-subset of $\Omega$ such that $|\Delta^{G}|>|\Sigma^{G}|$ for all $4$-subsets $\Sigma$ containing $\Delta$. Then $\Delta^G$ is of maximal length of any orbit on $3$-subsets.

\end{proposition}
\begin{proof}[Proof of Proposition~\ref{maxsize}]
Let $\Delta=\Delta_1$. If $G$ is $3$-homogeneous we are done, so assume there exists a $3$-subset $\Delta_2$ which is not in $\Delta_1^G.$ We wish to show that $|\Delta_1^G|\ge |\Delta_2^G|$. Theorem~\ref{SW1} tells us that every $2$-subset of $\Omega$ appears in the $G$-image of $\Delta_1,$  Hence we may assume that $|\Delta_1 \cap \Delta_2|=2.$ We may now set $\Delta_1 = \{1,2,3\}$ and $\Delta_2 = \{1,2,4\}.$ We may now choose $\Sigma =\{1,2,3,4\}$ and note that $G_\Sigma$ is not transitive on $\Sigma$ and that $|\Sigma^G|<|\Delta_1^G|.$ 

We now consider the possible structure of $G_\Sigma$ as a subgroup of $Sym(4).$ Theorem~\ref{SW1} tells us that 
$$4\ge |\Delta_1^{G_\Sigma}|>|\Sigma^{G_{\Delta_1}}| \ge 1,$$
and moreover $|\Delta_1^{G_\Sigma}|= 2$ or $3$.

By applying Lemma~\ref{ud lemma} and noting that $|\Delta_1^G|>|\Sigma^G|$ we have the following equations.

\begin{align*} 
u_1|\Delta_1^G| &= d_1|\Sigma^G|, \\ 
u_2|\Delta_2^G| &= d_2|\Sigma^G| .\\
\end{align*}

We also note that $1\le u_1 < d_1 \le 3$ and $u_1, u_2, d_1, d_2 >0.$

If $|\Delta_1^{G_\Sigma}|=3$ then $d_1=3$ giving us $u_1|\Delta_1^G| = 3|\Sigma^G|.$ This implies that $d_2=1$. Combining these gives
$$|\Delta_2^G| \le u_2|\Delta_2^G| = |\Sigma^G| =\frac{u_1|\Delta_1^G|}{3} < |\Delta_1^G|.$$

Next we consider the case when $|\Delta_1^{G_\Sigma}|=2$. Here we have $d_1=2$ and so $2|\Sigma^G|=u_1|\Delta_1^G|$. It immediately follows that $u_1=1$. Moreover $d_2|\Sigma^G|=u_2|\Delta_2^G|$ where $1\le d_2\le 2.$ Now
\begin{align*}
u_2|\Delta_2^G| & = d_2|\Sigma^G|,\\
u_2|\Delta_2^G| & = \frac{d_2}{2}|\Delta_1^G|,\\
\frac{2u_2}{d_2} |\Delta_2^G| & = |\Delta_1^G|.
\end{align*}
where $2u_2 \geq d_2$ and so $|\Delta_1^G| \ge |\Delta_2^G|$ as required.

\end{proof}

\begin{proposition}
Let $G$ be a primitive permutation group satisfying $\star$. Let $\Delta$ be a representative of the large orbit on $\Omega_3$ and let $G_\Delta$, the set-wise stabilizer of $\Delta$ in $G,$ be transitive on $\Delta$. Then $G$ is $3$-homogeneous.
\end{proposition}

\begin{proof}
We assume $G$ is not $3$-homogeneous, then we let $\Delta=\Delta_1$ and $\Delta_2$ be a representative from a second $G$-orbit on $\Omega_3$. By Theorem~\ref{SW1} we can assume that $\Delta_1=\{1,2,3\}$ and $\Delta_2=\{1,2,4\}$. We also let $\Sigma=\{1,2,3,4\}$. As $G_{\Delta_1}$ is transitive on $\Delta_1$ there exists $g \in G_{\Delta_1}$ such that $g=\pi_1 \pi_2 \pi_3$ where $\pi_1=(1,2,3),$ $\pi_2=(4,...,)$ (the cycle containing $4$) and $\pi_3$ is a permutation on the remaining points of $\Omega$. Now as $\Delta_1^{G_{\Sigma}}>1$ there exists a subset of $\Sigma$ of the form $\{\alpha, \beta, 4\}$ which is in $\Delta_1^G$. If $\pi_2=(4)$ then $\Delta_2 \in \Delta_1^{\left<g\right>}$ which is a contradiction. Without loss then we can assume $g$ is such that $\pi_2$ is a cycle with length some power of $3$. We then only need consider two cases. 
\begin{enumerate}
\item $g=(1,2,3)(4,5,6)...$ or

\item $g=(1)(2)(3)(4,5,6)...$
\end{enumerate}

In both cases we have that $\{1,2,3,4\}, \{1,2,3,5\}, \{1,2,3,6\}$ are all in $\Sigma^{\left<g\right>}$. Using Lemma~\ref{ud lemma} we have $u|\Delta_1^G| = d|\Sigma^G|$ with $u\ge 3$ giving us that $d=4$ which is only possible if $\Delta_2 \in \Delta_1^G$ which is a contradiction. Hence $G$ is $3$-homogeneous as required.

\end{proof}

\section{The case that $G$ is $3$-homogeneous}

The list of Primitive groups stored in \textsc{Magma} and the procedure used in Proposition~\ref{PPPSW} were used to check for any instances of the Siemons Wagner property for $k=3$ in all primitive groups of degree $n\leq 200$. The only two groups found are given in Proposition~\ref{PPPSW}. As both of these groups are $3$-homogeneous the obvious question is are there any more.

\begin{theorem}\label{3hom}
Let $G$ be a $3$-homogeneous permutation group acting on a $G$-set, $\Omega$, of cardinality $n\ge8$. If the $G$-orbit $\Omega_3$ has length strictly greater than the $G$-orbit of any $4$-subset of $\Omega$ then $G\cong L_2(7)$ or $G\cong PGL(2,7)$.
\end{theorem}

Before we prove this we mention that the following notation for $PG(q)$ is used.
\begin{definition}
The projective line over $\mathbb{F}_q$ contains $q+1$ points and we denote it by $PG(q)$. The composition of these points and their representatives for $\mathbb{F}_q$ is as given in Table \ref{Tab2.1}.

\begin{table}[ht]
\centering
\begin{tabular}[h]{|l|l|}
  \hline
 Code & Span of\\ \hline
  $0$ & $ (1,0)$ \\
  $1$ & $  (1,1)$ \\
    $\vdots$ & $  \vdots$ \\
  $q-1$ & $ (1,q-1) $\\
  $\infty$ & $  (0,1) $\\

  \hline

\end{tabular}
\caption{Elements of Projective Line over $\mathbb{F}_q$}
\label{Tab2.1}
\end{table}
\end{definition}

\begin{proof}[Proof of Theorem~\ref{3hom}]
As $G$ is 3-homogeneous we note that any group hoping to have the property we are searching for cannot have any regular orbits on $\Omega_4$ and furthermore cannot be $4$-homogeneous.

We begin by compiling a list of possible candidates using results from Kantor~\cite{Kantor} and Cameron~\cite{CamFPG}. This gives the following possibilities $M_{11}$, $M_{22}$, $AGL(1,8)$, $A\Gamma L(1,8)$, $A\Gamma L(1,32),$ $L_2(q)$ and a family of groups $L_2(q)\le G \le P\Gamma L(2,q)$, where $q \equiv\ 3$ mod $4$.

Initially we can compute the groups $M_{11}$, $M_{22}$, $AGL(1,8)$, $A\Gamma L(1,8)$ and $A\Gamma L(1,32)$ and see that these do not satisfy the condition of having such a $3$-subset of their respective $G$-sets.

We now eliminate the possibilities for $q$ in the remaining families of groups. We begin by noting that all the groups which remain have $L_2(q)$ as their respective socles. As the $G$-sets of these groups are the same as for their socles, we need only show that $L_2(q)$ does not satisfy the condition of having all $4$-subsets with stabilizers of order greater than $2$. This follows as the size of the orbit on $3$-subsets being as large as possible but less than $\frac{|G|}{2}$ and so any over group cannot increase the length of this $3$-orbit but could fuse two or more orbits on $\Omega_4$.

Next we consider a specific subset of the projective line on which these remaining groups act.

We let $\omega$ be a generator for the multiplicative field of $q$ elements, that is $\omega^{q-1}=1$. We choose $\Sigma=\{0,1,\infty,\omega^{a}\}$ where $\omega^a \not\in \{-1,2,2^{-1}\}$ and, $\omega^{2a}-\omega^{a}+1\ne 0$.

We can now make use of a result in~\cite{bjl} which states that for such a set the stabilizer in $PGL(2, q)$ has order $4$. We now show that this set must have a stabilizer in $L_2(q)$ with order less than 4 by giving an element of $PGL(2,q)_\Sigma$ which is not in $L_2(q)$.

$$A=\begin{bmatrix} 1 & \omega^{a} \\ -1 & -1 \ \end{bmatrix}.$$

It is clear that $A$ will act on $\Sigma$ with the cycles $(0,\omega^{a})(1,\infty)$ and so $A\in PGL(2,q)_\Sigma.$
The determinant of $A$ is equal to $\omega^{a}-1$, moreover $\omega^{a} \neq 2$ by choice and so $A\not\in L_2(q)$ as required. Hence the stabilizer of $\Sigma$ in $L_2(q)$ must have order $1$ or $2$ and so the $L_2(q)$ orbit of $\Sigma$ is greater than the total number of $3$-subsets of $PG(q)$ hence such groups cannot have a large enough $3$-subset orbit to satisfy our condition.

This has now reduced our problem to finding fields for which no such element $\omega^a$ exists. We also note that we are interested in $q\ge 7$. In fact $q=7$ is the only such field in our range without such an element as a simple counting argument shows that any field with more than 8 elements must satisfy this requirement.

Finally we note that $L_2(7)<PGL(2,7) = P\Gamma L(2,7)$ and that Proposition~\ref{PPPSW} shows both of these groups satisfy the condition.

\end{proof}

We leave this section with the following Conjecture.

\begin{conjecture}
Let $G$ be a primitive permutation group acting on a $G$-set, $\Omega$, of cardinality $n\ge8$. If there exists a $3$-subset $\Delta\subset\Omega$ such that $|\Delta^G|>|\Sigma^G|$ for any $4$-subset $\Sigma$ containing $\Delta$, then $G\cong PSL(2,7)$ or $G\cong PGL(2,7)$.
\end{conjecture}

\section{Non-primitive group examples}

For the reader's consideration we now turn briefly to non-primitive but transitive examples of groups with the Siemons Wagner property when $k=3.$ These are far more common than imprimitive examples. We present here three such groups and give their generators. In all three of the following examples we keep $\Delta=\{1,2,3\}$ as a representative of the large orbit on $\Omega_3$. For a given permutation group $G$ acting on a set $\Omega$, we denote the number of $G$-orbits on $\Omega_k$ by $\sigma_k$.

Example 1 is a subgroup of $Sym(8)$ where

\begin{eqnarray*}
G_1 & \cong & \langle(4,6),(1,2,5,3)(4,8)(6,7),(1,8)(4,6),(3,4,6),(1,7,8),(2,3)(4,6), \\
  & &(2,4)(3,6),(1,5)(7,8),(1,7)(5,8)\rangle.\\
\end{eqnarray*}

Here $|G_1|=1152$ and $|\Delta^{G_1}|=48$. This group satisfies the condition that every $2$-subset appears in some $G_1$-image of $\Delta$. The system of imprimitivity for $G_1$ is the set $$\left\{\{1,5,7,8\},\{2,3,4,6\}\right\}.$$ The $\sigma_k$ values are $\sigma_1=1,\ \sigma_2=2,\ \sigma_3=2$ and $\sigma_4=3$.

It is also clear that there exists a $4$-subset for which no $G_1$-image contains $\Delta$ as a subset (the system of imprimitivity is a single orbit).

Example 2 is a subgroup of $Sym(9)$ where
\begin{eqnarray*}
G_2 & \cong & \langle(4,7)(5,9)(6,1),(8,9,5)(2,7,4)(3,1,6), \\ 
 & & (4,5,6)(7,1,9),(8,3,2)(7,1,9)\rangle.
\end{eqnarray*}

Here $|G_2|=54$ and $|\Delta^{G_2}|=54$. This group satisfies the condition that every $2$-subset appears in some $G_2$-image of $\Delta$. The system of imprimitivity for $G_2$ is the set $$\{\{1,7,9\},\{2,3,8\},\{4,5,6\}\}.$$
The $\sigma_k$ values are $\sigma_1=1,\ \sigma_2=2,\ \sigma_3=5$ and $\sigma_4=5$. In this case $\Delta$ appears as a subset of an element of every $G_2$-orbit on subsets of size $4.$

Example 3 is a subgroup of $Sym(16)$ where
\begin{eqnarray*}
G_3 & \cong & \langle  (1, 12)(7, 3)(11, 8)(4, 2)(5, 10)(6, 9)(13, 15)(14, 16),\\
  & &(1, 8, 6, 14)(7, 2, 5, 13)(11, 9, 16, 12)(4, 10, 15, 3),\\
    & &(1, 14)(7, 13)(11, 12)(4, 3)(5, 2)(6, 8)(9, 16)(10, 15),\\
  & & (1, 6)(7, 5)(11, 15)(4, 16)(2, 14)(8, 13)(9, 12)(10, 3),\\
 & & (1, 16)(7, 15)(11, 6)(4, 5)(2, 3)(8, 12)(9, 14)(10, 13),\\
 & &  (7, 8)(9, 10)(3, 12)(13, 14),\\
 & & (11, 4)(9, 10)(3, 12)(15, 16),\\
 & & (1, 7)(11, 4)(5, 6)(2, 8)(9, 10)(3, 12)(13, 14)(15, 16)\rangle.\\
\end{eqnarray*}

Here $|G_3|=256$ and $|\Delta^{G_3}|=256$. This group does not satisfy the condition of every $2$-subset being contained in some $G_3$-image of $\Delta$. The $\sigma_k$ values for $G_3$ are $\sigma_1=1,$ $\sigma_2=6,$ $\sigma_3=11,$ $\sigma_4=35,$ $\sigma_5=48,$ $\sigma_6=91,$ $\sigma_7=100$ and $\sigma_8=132.$ Here we have three systems of imprimitivity
 $$ \{  \{ 1, 5 \},
    \{ 2, 14 \},
    \{ 3, 9 \},
    \{ 4, 16 \},
    \{ 6, 7 \},
    \{ 8, 13 \},
    \{ 10, 12 \},
    \{ 11, 15 \}\},$$

$$\{   \{ 3, 9, 10, 12 \},
    \{ 1, 5, 6, 7 \},
    \{ 4, 11, 15, 16 \},
    \{ 2, 8, 13, 14 \} \}$$

    and
$$   \{ \{ 1, 3, 5, 6, 7, 9, 10, 12 \},
    \{ 2, 4, 8, 11, 13, 14, 15, 16 \}\}.$$

It is also clear that there exists a $4$-subset for which no $G_3$-image contains $\Delta$ as a subset (the system of imprimitivity is a single orbit). 

\section{\textsc{Magma} Code for Proposition~\ref{PPPSW}}

We finish by giving the \textsc{Magma} implementation for the search for Primitive groups with the Siemons Wagner property used to compile Table~\ref{PPSW}

 \begin{verbatim}
Z:=Integers();

SizeofOrbsPRIMk:=procedure(G,k,~a);
S:={}; K:={}; D:={1..Degree(G)}; 
	kD:=Subsets(D,k);a:={};
Omega:=GSet(G,kD);
	O:=Orbits(G,Omega);
	for Orbs in O do 
		T:=Random(Orbs);	
		Include(~K,T);
	end for;
V:={};
	for T in K do;
		N:=Z!(#G/#Stabilizer(G,T));
			P:=D diff T;
				for b in P do;
				Include(~V, #Stabilizer(G, T join{b}));
				end for;
			S:=Min(V);L:= Z!(#G/S); 
				if N gt L then Include(~a,<N,T>);
				end if;
	end for;
end procedure;
\end{verbatim}
Letting $D$ be the degree.
\begin{verbatim}
for k in [3..Z!(Floor(D/2)-1)] do
	for I:=1 to (Z!(NumberOfPrimitiveGroups(D)-2)) do
		SizeofOrbsPRIMk(PrimitiveGroup(D, I),k,~T);
			if #T ge 1 then <D,I,T>;
			end if;
	end for;
end for;
\end{verbatim}

\appendix

\end{document}